\documentclass[11pt,english]{article}
\usepackage{amsmath, amsthm, latexsym, amssymb,url,dsfont}
\usepackage{amsfonts}
\usepackage{epsfig}
\usepackage{enumerate}
\usepackage{graphicx}
\usepackage{yhmath}
\usepackage{mathdots}
\usepackage{mathtools} 
\usepackage{pifont}

\usepackage{tikz}
\usetikzlibrary{matrix,arrows,decorations.pathmorphing}
\usepackage{tikz-cd}

\usepackage{relsize}

\newcommand{\shrinkmargins}[1]{
  \addtolength{\textheight}{#1\topmargin}
  \addtolength{\textheight}{#1\topmargin}
  \addtolength{\textwidth}{#1\oddsidemargin}
  \addtolength{\textwidth}{#1\evensidemargin}
  \addtolength{\topmargin}{-#1\topmargin}
  \addtolength{\oddsidemargin}{-#1\oddsidemargin}
 \addtolength{\evensidemargin}{-#1\evensidemargin}
  }

\shrinkmargins{.5}

\theoremstyle{plain}

\newtheorem{theorem}{Theorem}[section]
\newtheorem{corollary}[theorem]{Corollary}
\newtheorem{conjecture}[theorem]{Conjecture}
\newtheorem{lemma}[theorem]{Lemma}
\newtheorem{proposition}[theorem]{Proposition}
\newtheorem{question}[theorem]{Question}

\newtheorem*{teo}{Theorem}
\newtheorem*{conje}{Conjecture}

\newtheorem{definition}[theorem]{Definition}

\theoremstyle{remark}
\newtheorem{remark}[theorem]{Remark}

\theoremstyle{definition}
\newtheorem{example}[theorem]{Example}

\def \Z {\mathbb{Z}}
\def \Q {\mathbb{Q}}

\def \R { \mathbb{R}}
\def \det { \text{det}}

\def \d { \text{Disc}}
\def \tr { {\rm tr}}

\begin{document}

\thispagestyle{empty}
\setcounter{tocdepth}{7}

\title{A proof of a conjecture on trace-zero forms and shapes of number fields}
\author{Guillermo Mantilla-Soler \and Carlos Rivera-Guaca}

\date{}

\maketitle

\begin{abstract}

In 2012 the first named author conjectured that totally real quartic fields of fundamental discriminant are determined by the isometry class of the integral trace zero form; such conjecture was based on computational evidence and the analog statement for cubic fields which was proved using Bhargava's higher composition laws on cubes. Here, using Bhargava's parametrization of quartic fields we prove the conjecture by generalizing the ideas used in the cubic case. Since at the moment, for arbitrary degrees, there is nothing like Bhargava's parametrizations we cannot deal with degrees $n > 5$ in a similar fashion. Nevertheless, using some of our previous work on trace forms we generalize this result to higher degrees;  we show that if  $n \ge 3$ is an integer such that $(\Z/n\Z)^{*}$ is a cyclic group, then the shape is a complete invariant for totally real degree $n$ number fields with fundamental discriminant.
\end{abstract}

\section{Introduction}

Let $K$ be a number field of degree $n:=[K:\mathbb{Q}]$  and let $O_{K}$ be its maximal order. The {\it trace zero module} of $O_{K}$ is the $\Z$-submodule of $O_{K}$ given by the kernel of the trace map i.e., $O_{K}^{0}=K^{0} \cap O_{K}$ where $K^{0}:= \{ x \in K : \tr_{K/\mathbb{Q}}(x)=0 \}.$ The {\it integral trace-zero form} of $K$ is the isometry class of the rank $n-1$ quadratic $\Z$-module $(O_{K}^{0}, \tr_{K/\Q})$ given by restricting the trace pairing from $O_{K} \times O_{K}$  to $O_{K}^{0} \times O_{K}^{0}.$  It is clear that the isometry class of the quadratic module $(O_{K}^{0}, \tr_{K/\Q})$ determines the field $K$ for $n=1,2$. This is not the case for cubic fields, see \cite[\S 3]{Man}. However, it can be shown, using  Delone-Faddeev-Gross parametrization of cubic rings and Bhargava's higher composition laws on cubes,  that for totally real cubic fields of fundamental discriminant $(O_{K}^{0}, \tr_{K/\Q})$ determines the field (see \cite[Theorem 6.5]{Man}).  In \cite{Man2} the first named author conjectured that the above property of the trace zero form is not only particular of degrees less than 4 but also works for quartic fields (see  \cite[Conjecture 2.10]{Man2}). In Section \S \ref{QuarticParametrization} we prove such conjecture via Bhargava's parametrization of quartic rings:

\begin{teo}[cf. \S \ref{QuarticParametrization}]
Let $K$ be a totally real quartic number field with fundamental discriminant. If $L$
is a tamely ramified number field such that an isomorphism of quadratic modules
\[(O_{K}^{0}, \tr_{K/\Q}) \cong (O_{L}^{0}, \tr_{L/\Q})\] exists, then $K \cong L$.
\end{teo}

Another quadratic invariant, with a more geometric interpretation and closely related to the trace zero form, that has been studied by several authors is the shape of $K$.  Endow $K$ with the real-valued  $\mathbb{Q}$-bilinear form $b_K$ whose associated quadratic form is given by \[b_K(x,x):= \sum_{\sigma : K \hookrightarrow \mathbb{C}}  |\sigma(x)|^2.\] The \textit{shape} of $K$, denoted $\textrm{Sh}(K)$, is  the isometry equivalence class of $ (O_ K^{\bot},b_K)$ up to scalar multiplication, where $O_ K^{\bot}$ is the image of $O_K$ under the projection map, 
 $\alpha \mapsto \alpha_{\bot}:=n\alpha-\tr_{K/\mathbb{Q}}(\alpha)$, i.e.,
\[O_ K^{\bot}:= \{\alpha_\bot : \alpha \in O_K \}=(\mathbb{Z}+n O_K) \cap O_{K}^{0}.\]
Thus $\textrm{Sh}(K)=\textrm{Sh}(L)$ if and only if $(O_K^{\bot},b_K) \cong  (O_L^{\bot}, \lambda b_L) $ for some $\lambda \in \mathbb{R}^\times$. Equivalently,  ${\rm Sh}(K)$ can be thought as the  $(n-1)$-dimensional lattice inside $\R^{n}$, via the Minkowski embedding, that is the orthogonal complement of $1$ and that is defined up to reflection, rotations and scaling by $\R^{*}$. Hence {\rm Sh}(K) correspond to an element to the {\it space of shapes} \[\mathcal{S}_{n-1}:= {\rm GL}_{n-1}(\Z) \setminus {\rm GL}_{n-1}(\R)/{\rm GO}_{n-1}(\R).\] The distribution of shapes of number fields in $\mathcal{S}_{n}$ have been the subject of a lot of interesting current research (see \cite{bhargavaPh,RobH,RobH1, Man3}). Our main result about shapes and trace zero forms is the following:

\begin{teo}[cf. Theorem \ref{Main}]

Let $K$ be totally real number field of fundamental discriminant and degree $n\ge 3$. If  $ \left(\mathbb{Z}/n \mathbb{Z}\right)^\times$ is cyclic, then for any number field $L$ the following are equivalent:
\begin{enumerate}[(i)]
\item $K \cong L$.

\item $(O_K^{\bot},\tr_{K/\mathbb{Q}})  \cong (O_L^{\bot},\tr_{L/\mathbb{Q}})$.

\item ${\rm Sh}(K)={\rm Sh}(L)$ and $L$ is totally real with fundamental discriminant.
\end{enumerate} If $\left(n,\d(K)\right)=1$ then the three items are also equivalent to (iv) $(O_K^{0},\tr_{K/\mathbb{Q}})  \cong (O_L^{0},\tr_{L/\mathbb{Q}})$.

\end{teo}

Based in the above theorem, the results proved here, Question \ref{Question8} and the computational evidence given in \cite[Table 1]{Man2} we propose the following conjecture:

\begin{conjecture}
Let $K$ be a totally real octic number field with fundamental discriminant. If $L$
is a tamely ramified number field such that an isomorphism of quadratic modules
\[(O_{K}^{0}, \tr_{K/\Q}) \cong (O_{L}^{0}, \tr_{L/\Q})\] exists, then $K \cong L$.
\end{conjecture}

\subsection{Structure of the paper}

In \S \ref{ElGeneral} we prove our main results on trace zero modules and shapes of arbitrary degree. In \S \ref{CasimirInv} we briefly remind some of the basic definitions of Casimir invariants, and in \S \ref{QuarticParametrization} we give a proof of  \cite[Conjecture 2.10]{Man2} based on Bhargava's parametrization of quartic rings.

\section{From $O_{K}^{0}$ and $O_K^{\bot}$ to $O_{K}$}\label{ElGeneral}

In this section we study in detail the structure of the modules $O_{K}^{0}$ and $O_K^{\bot}$ and layout a strategy to see when an isometry between the trace zero parts (resp. shapes) can be lifted to a full isometry of the integral trace. By the end of this section we explain what are the limitations of such strategy.

\subsection{The discriminants of $O_{K}^{0}$ and $O_K^{\bot}$.}

We start by finding out what are the values of the determinants of the quadratic modules $O_{K}^{0}$ and $O_K^{\bot}$.\\

We denote by  ${\rm Disc}(O_K^{0})$ (resp. ${\rm Disc}(O_K^{\bot})$)  the determinant of the quadratic module $(O_K^{0},\tr_{K/\mathbb{Q}})$ (resp. $(O_K^{\bot},\tr_{K/\mathbb{Q}})$).

The following lemma will be of use to us.

\begin{lemma}\label{DetsRelativos}
Let $K$ be a degree $n>1$ number field and let $k$ be a positive integer such that $\tr_{K/\Q}(O_{K})=k\Z$.  Then, \[{\rm Disc}(O_K^{0}) = \frac{n}{k^2}{\rm Disc}(K) \ {\rm and} \  {\rm Disc}(O_K^{\bot})=n^{2n-3} {\rm Disc}(K). \]
\end{lemma}

\begin{proof}

Using the surjectivity of trace map in the instances $\tr_{K/\Q} :O_{K} \twoheadrightarrow k\Z$
 and $\tr_{K/\Q} :\Z +nO_{K} \twoheadrightarrow n\Z$
 we observe that \[\left(O_{K} / (\Z + O_K^{0}) \right) \cong \Z/ \left( n/k\right)\Z  \  {\rm and}  \ \Z +nO_{K} = \Z + O_K^{\bot}.\] Since $O_{K}$ has an integral basis containing $1$ there is an isomorphism of $\Z$-modules $O_{K}/(\Z +nO_{K}) \cong (\Z/n\Z)^{n-1}$.  In particular,  \[ [O_{K} : \Z + O_K^{0}] = n/k  \  {\rm and}  \ [O_{K} : \Z + O_K^{\bot}] =n^{n-1}.\] The result follows from this and from the fact that the decomposition $\Z + O_K^{0}$ is orthogonal with respect the trace pairing. 

\end{proof}

\begin{lemma} (Maurer, \cite{Mau}).\label{TraceMau} Let $k$ be an integer such that $\tr_{K/\mathbb{Q}}(O_K) = k\mathbb{Z}$, then a prime $p$ divides $k$ if and only if $p \mid e(\mathfrak{p}|p) $ for all primes $\mathfrak{p}$ in $K$ lying over $p$.
\end{lemma}

\begin{proof}
We will give an alternative proof to that presented in \cite{Mau}. In fact, it is not hard to see that the following slightly more general statement holds: Let $\mathcal{D}_{K/\mathbb{Q}}$ denote the different ideal of $K/\mathbb{Q}$, then for a rational prime $p$ we have
\[v_p(k)=\min_{\mathfrak{p} \mid \,  p} \left \lfloor{\frac{v_{\mathfrak{p}}\left(\mathcal{D}_{K/\mathbb{Q}} \right)}{e(\mathfrak{p}| p)}}\right \rfloor,\]

\noindent and the lemma will follow from the fact $v_{\mathfrak{p}}\left(\mathcal{D}_{K/\mathbb{Q}} \right) \geq e(\mathfrak{p}| p) \iff p \mid e(\mathfrak{p}| p)$. Since the trace map is a morphism of $\Z$-modules, and by hypothesis its image is  $k\mathbb{Z}$, the result is obtained from the following elementary equivalences. For any $r\in \mathbb{Z}^+ \cup \{0\} $,
\begin{equation*}
\begin{split}
 p^r \mid k & \iff \tr_{K/\mathbb{Q}}(O_K) \subset p^r \mathbb{Z}\\
 & \iff \tr_{K/\mathbb{Q}}(p^{-r}\,O_K) \subset \mathbb{Z}\\
 & \iff p^{-r} \in \mathcal{D}_{K/\mathbb{Q}}^{-1} \\ 
  & \iff  \mathcal{D}_{K/\mathbb{Q}} \subset p^{r}O_{K}\\
   & \iff  r  v_{\mathfrak{p}}(p) \leq v_{\mathfrak{p}}\left( \mathcal{D}_{K/\mathbb{Q}} \right),  \textnormal{ for all } \mathfrak{p} \textnormal{ prime in } K\\
    & \iff   r  e(\mathfrak{p}|p) \leq v_{\mathfrak{p}}\left( \mathcal{D}_{K/\mathbb{Q}}\right),  \textnormal{ for all } \mathfrak{p} \textnormal{ dividing } p \\
    & \iff   r \leq  \left \lfloor{\frac{v_{\mathfrak{p}}\left(\mathcal{D}_{K/\mathbb{Q}} \right)}{e(\mathfrak{p}| p)}}\right \rfloor,   \textnormal{ for all } \mathfrak{p} \textnormal{ dividing } p \\  
    & \iff r \leq  \min_{\mathfrak{p} \mid \,  p} \left \lfloor{\frac{v_{\mathfrak{p}}\left(\mathcal{D}_{K/\mathbb{Q}} \right)}{e(\mathfrak{p}| p)}}\right \rfloor.
\end{split}
\end{equation*}
The result follows, since $v_p(k)$ is the largest of such $r$.
\end{proof}
\vspace*{0.5cm}
\begin{corollary}\label{TraceZ}
Let $L$ be a number field of degree $n$. Suppose any of the following conditions holds:\begin{enumerate}[(i)]
\item Every prime $p \mid n$ satisfies $p^n \nmid \d(L). $
\item The extension $L/\mathbb{Q}$ is tamely ramified at $p$, for every $p\mid n$.
\item The field $L$ has fundamental discriminant and $n >2$.
\end{enumerate}
Then, $\tr_{L/\mathbb{Q}}(O_L) = \mathbb{Z}$. Moreover, if $n=4$  condition $(i)$ is equivalent to $\tr_{L/\mathbb{Q}}(O_L) = \mathbb{Z}$ and if $n$ is a  prime conditions $(i)$ and $(ii)$ are equivalent to $\tr_{L/\mathbb{Q}}(O_L) = \mathbb{Z}$.
\end{corollary}

\begin{proof}
Let $k \in \mathbb{Z}^+$ be such that $\tr_{L/\mathbb{Q}}(O_L)=k\mathbb{Z}$, then:

\begin{itemize}
\item If $(i)$ holds, by definition of discriminant as determinant of a Gram matrix of $(O_L,\tr_{L/\mathbb{Q}})$, we know that $k^n \mid \d(L)$, as $k \mid n$, any $p$ prime dividing $k$ would satisfy $p^n \mid \d(L)$  and $p \mid n$, thus no such prime exists and $k=1$.
\item If $(ii)$ holds and $p$ is prime dividing $k$ then $p$ divides $n$; this since $n=\tr_{L/\mathbb{Q}}(1)$  and therefore $k \mid n$. Then $p \nmid e(\mathfrak{p}| p)$ for every prime $\mathfrak{p}$ lying over $p$, contradicting Lemma \ref{TraceMau}. Therefore $k=1$.
\item If $(iii)$ holds, the extension $L/\mathbb{Q}$ is tame at every odd prime,  hence $p \nmid k$ for all $p \neq 2$. Suppose $2 \mid k$ and let $f_1^{e_1} \, f_2^{e_2} \cdots $   be the factorization type of $2$ in $L$, then $2 \mid e_i$ for all $i$, and  
\[ 3 \geq v_2(\d(L)) \geq e_1 f_1 +e_2 f_2 +\cdots=n > 2  \]
yields  $n=3=e_1f_1+e_2f_2+\cdots \equiv 0 \mod 2$, a contradiction. Thus $k=1$.
\end{itemize}

To prove the reciprocals, suppose  $k=1$ and $n=p$ is prime. Let $f_1^{e_1}  \cdots f_g^{e_g}$   be factorization type of $p$ in $L$. If condition $(ii)$ did not hold, there would be an index $i$ such that $p \mid e_i$ which we may assume is $i=1$. In that case $p=e_1f_1+\cdots \geq p$ implies that there is only one prime $\mathfrak{p}$ in $L$ lying over $p$ and for that prime $p \mid e(\mathfrak{p} | p)=e_1$, contradicting $(\ref{TraceMau})$. Since condition $(ii)$ always implies condition $(i)$, this proves \[\tr_{L/\mathbb{Q}}(O_L) = \mathbb{Z}\iff (ii) \iff (i),\] in this case.  \\

Similarly, if $k=1$, $n=4$ and condition $(i)$ did not hold, then $2^4 \mid \d(L)$, in particular $2$ would be wildly ramified in $L$, so its factorization type would be $1^2 2$, $1^2 1 1$, $1^2 1^2$ or $2^2$. Since $k=1$, only the first two cases are possible. However, for a prime $\mathfrak{p}$ in $K$ the exact power of $\mathfrak{p}$ dividing the different ideal $\mathcal{D}_{L/\mathbb{Q}}$ is bounded by  $e\cdot(v_p(e)+1)-1$, where $e=e(\mathfrak{p}|p)$ and $p=\mathfrak{p} \cap \mathbb{Z}$, thus in either case if $\mathfrak{p}$ is the prime in $L$ with $e(\mathfrak{p}|2)=2 $  we would have $\mathcal{D}_{L/\mathbb{Q}}=\mathfrak{p}^v\mathfrak{a}$, where $(\mathfrak{a},2)=1$ and $v \leq 3$. This yields $ 4 \leq v_2(\d(L))=v_2\left(\mathcal{N}(\mathcal{D}_{L/\mathbb{Q}}) \right)=v\cdot f(\mathfrak{p}\mid 2)=v\leq3  $, a contradiction.
\end{proof}

We say that a number field $L$ is tame, or tamely ramified, if no rational prime ramifies wildly in $L$.

\begin{lemma}\label{TameYFundaIsSquafree}

Let $n >1$ be an integer and let $K$ and $L$ be two degree $n$ number fields. Suppose that $K$ has fundamental discriminant and that $L$ is tame. If  \[(O_{K}^{0}, \tr_{K/\Q}) \cong (O_{L}^{0}, \tr_{L/\Q})\] then, ${\rm Disc}(K)={\rm Disc}(L)$ and moreover such discriminant is square free. In particular, a tame number field with fundamental discriminant has square free discriminant.
\end{lemma}

\begin{proof}
Let us assume first that $n >2$. In such case it follows from Corollary \ref{TraceZ} that $\tr_{K/\mathbb{Q}}(O_K) = \mathbb{Z}=\tr_{L/\mathbb{Q}}(O_L)$. Thus ${\rm Disc}(K)={\rm Disc}(L)$ thanks to Lemma \ref{DetsRelativos}. Let $\widetilde{L}$ be the Galois closure of $L/\Q$. Recall that tame extensions are closed under composition and sub-extensions, see \cite[Chapter II, Corollaries 7.8, 7.9]{Ne}. Hence since $L$ is tame so it is $\widetilde{L}$, and so it is any $E/\Q$ sub-extension of $\widetilde{L}/\Q$. If ${\rm Disc}(L)$ were not square free, and since it is a fundamental discriminant, the extension $E=\Q(\sqrt{{\rm Disc}(L)})/\Q$ would be a sub-extension $\widetilde{L}/\Q$ that is not tame; since it has wild ramification at $p=2$. Now suppose $n=2$. If ${\rm Disc}(K) \neq {\rm Disc}(L)$ then, by Lemma \ref{DetsRelativos}, ${\rm Disc}(K)=4{\rm Disc}(L)$ which implies, since $n=2$, that $K \cong L$ thus ${\rm Disc}(K)={\rm Disc}(L)$ contradicting the initial hypothesis.
\end{proof}

\subsection{A basis for $O_{K}^{0}$ and a condition of an isometry extension.}

For a totally real number field $K$ \cite[Lemma 5.1]{Casimir} shows that $(O_K,\tr_{K/\mathbb{Q}}) \cong (O_L,\tr_{L/\mathbb{Q}})  $ always implies $(O_K^0,\tr_{K/\mathbb{Q}}) \cong (O_L^0,\tr_{L/\mathbb{Q}})  $. The following example (found by `brute force" using \texttt{Magma} ) shows that the converse is not true.

\begin{example}
Let $K$ and $L$ be the quartic fields with defining polynomials $x^4+82x^2+656$ and $x^4-2x^3-19x^2+20x+18$ respectively, then $K$ and $L$ are totally real fields such that $\d(K)=2^6 41^3=\d(L)$ and $(O_K^0,\tr_{K/\mathbb{Q}}) \cong (O_L^0,\tr_{L/\mathbb{Q}})  $ however $(O_K,\tr_{K/\mathbb{Q}}) \ncong (O_L,\tr_{L/\mathbb{Q}})  $  and $(O_K^{\bot},b_K) \ncong (O_L^{\bot},b_L)$.
\end{example}

In view of such example it is natural to ask under what conditions we could expect to have a reciprocal for \cite[Lemma 5.1]{Casimir}. To address that question, remark that since $K=\mathbb{Q}\,\bot \, K^0$, each isometry  $\varphi:(K^0,\tr_{K/\mathbb{Q}})\xrightarrow{\sim} (L^0,\tr_{L/\mathbb{Q}})$ has two natural extensions to an isometry $(K,\tr_{K/\mathbb{Q}})\cong (L,\tr_{L/\mathbb{Q}})$, the one taking $1$ to $1$ call it $\varphi^+$ and the one taking $1$ to $-1$ call it $\varphi^-$. These are in fact the only two possible extensions of $\varphi$. Indeed, given an isometry $\phi:(K,\tr_{K/\mathbb{Q}}) \xrightarrow{\sim}  (L,\tr_{L/\mathbb{Q}})$ extending $\varphi$, we know that $\phi(1)$ is orthogonal to $\phi(K^0)=L^0$, which is the orthogonal complement of $1\cdot\mathbb{Q}$ and, as $(L,\tr_{L/\mathbb{Q}})$ is non degenerate, this implies $\phi(1)\in \mathbb{Q}$ but then $[L:\mathbb{Q}]=\tr_{L/\mathbb{Q}}(\phi(1)^2)=[K:\mathbb{Q}]\phi(1)^2$, which proves $\phi(1) \in \{-1,+1\}$.\\

Since $K^0=O_K^{0}\cdot\mathbb{Q}=O_K^{\bot} \cdot \mathbb{Q}$, it follows that an isometry  \[\varphi:(O_K^0,\tr_{K/\mathbb{Q}})\xrightarrow{\sim} (O_L^0,\tr_{L/\mathbb{Q}}) \ {\rm resp.} \ \varphi:(O_K^{\bot},\tr_{K/\mathbb{Q}})\xrightarrow{\sim} (O_L^{\bot},\tr_{L/\mathbb{Q}})\]  will lift to an isometry  $(O_K,\tr_{K/\mathbb{Q}})\cong (O_L,\tr_{L/\mathbb{Q}})$ if and only if either $\varphi^{+}(O_K)=O_L$ or  $\varphi^{-}(O_K)=O_L$. But when do we have these equalities? This motivates the following:
\begin{lemma}\label{Lifting}
Let $L$ and $K$ be number fields and let $n:=[K:\mathbb{Q}]$, then:

\begin{enumerate}[(i)]

\item If $\varphi: (O_K^{\bot},\tr_{K/\mathbb{Q}})\xrightarrow{\sim} (O_L^{\bot},\tr_{L/\mathbb{Q}})$ is an isometry, then $\varphi^{\pm}(O_K)=O_L$ if and only if there exists a basis $\{1, \alpha_1 ,\ldots,\alpha_{n-1}\}$ of $O_K$ such that $t_i \equiv \pm s_i \mod n$ for all $1\leq i<n$, where $t_i:=\tr_{K/\mathbb{Q}}(\alpha_i)$ and the $s_i$'s are any integers such that $\varphi(\alpha_{i\bot})=n\beta_i-s_i \in O_L^{\bot}$,  with $\beta_i \in O_L$.
\item Suppose  that $\tr_{K/\mathbb{Q}}(O_K)=k\,\mathbb{Z}=\tr_{L/\mathbb{Q}}(O_L)$, $k \in\mathbb{Z}^+ $ and let  $\varphi: (O_K^0,\tr_{K/\mathbb{Q}})\xrightarrow{\sim} (O_L^0,\tr_{L/\mathbb{Q}})$ be an isometry, then $\varphi^{\pm}(O_K)=O_L$ if and only if $\varphi(\gamma_0) \equiv \pm 1 \mod n/k$. Where $\gamma_0:=1-(n/k) \gamma_K \in O_K^0$ and $\gamma_K$ is any element in $O_K$ such that $\tr_{K/\mathbb{Q}}(\gamma_K)=k$ .

\end{enumerate}

\end{lemma}

\begin{proof}  Note that in both cases the hypotheses and the existence of the respective isometry imply $[K:\mathbb{Q}]=n=[L:\mathbb{Q}]$ and $\d(K)=\d(L)$. For the degrees  this is clear, and for the discriminants it follows from the equalities, see Lemma \ref{DetsRelativos},  $\d(O_K^{\bot})=n^{2n-3}\d(K)$ and $\d(O_K^{0})=\frac{n}{k^2}\d(K)$. Thus in each case $\varphi^{\pm}(O_K) \subset O_L$ if and only if $\varphi^{\pm}(O_K) = O_L$. Now to prove  $(i)$ observe that for any basis $\{1, \alpha_1 ,\ldots,\alpha_{n-1}\}$ of $O_K$ we have

\begin{align*}
\varphi^{\pm}(O_K) \subset  O_L & \iff \varphi^{\pm}(\alpha_i) \in  O_L & \forall \, 1\leq i<n \\
& \iff \varphi^{\pm}\left(\frac{ \alpha_{i\bot}+ t_i}{n} \right)  \in  O_L &\forall \, 1\leq i<n \\
& \iff \frac{\varphi (\alpha_{i\bot}) \pm t_i}{n}  \in  O_L &\forall \, 1\leq i<n \\
& \iff \frac{n \beta_i-s_i \pm t_i}{n}  \in  O_L &\forall \, 1\leq i<n \\
& \iff \frac{-s_i \pm t_i}{n}  \in  O_L &\forall \, 1\leq i<n \\
&\iff s_i \equiv \pm  t_i \mod n &\forall \, 1\leq i<n 
\end{align*}

As for $(ii)$ observe that $O_K=\gamma_K \mathbb{Z}+O_K^0$, thus
\begin{align*}
\varphi^{\pm}(O_K) \subset  O_L & \iff \varphi^{\pm}(\gamma_K) \in  O_L \\
& \iff (k/n)\varphi^{\pm}(1-\gamma_0)  \in O_L\\
& \iff (k/n)( \pm 1-\varphi(\gamma_0))   \in  O_L\\
& \iff  \varphi(\gamma_0)=\pm 1\mod n/k
\end{align*} \end{proof}
To use this lemma, we begin by giving a description of the basis of $O_K^0$ that generalizes \cite[Proposition 5.2]{Man}

\begin{proposition}\label{BasisZero}
Suppose that $K$ is a number field of degree $n \geq 3$ and $\tr_{K/\mathbb{Q}}(O_K) =k \mathbb{Z}$, $k \in \mathbb{Z}^+$, then there exists a $\mathbb{Z}$-basis $\{1,\alpha_1, \ldots,\alpha_{n-1} \}$ of $O_K$ such that 
\[\left( \tr_{K/\mathbb{Q}}(\alpha_1), \ldots, \tr_{K/\mathbb{Q}}(\alpha_{n-2}),\tr_{K/\mathbb{Q}}(\alpha_{n-1}) \right)=:\left(t_1,\ldots,t_{n-2},t_{n-1} \right) \equiv (0,\ldots,0,k) \, \textrm{mod} \, n\]

and therefore $\{\alpha_1-t_1 /n, \ldots, \alpha_{n-2}-t_{n-2}/n,(n/k)\alpha_{n-1}-t_{n-1}/k\}$  is a basis of $O_K^0$.
\begin{remark}
If $n=2$, for any basis $\{1, \alpha_1\}$ of $O_K$ we have $\tr_{K/\mathbb{Q}}(\alpha_1)=t_1\equiv k \mod 2$ and $\{(2/k)\alpha_{1}-t_{1}/k\}$ is a basis of $O_K^{0}$.
\end{remark}

\end{proposition}

The proof  is based in the following elementary lemma:

\begin{lemma}\label{Elemt} Let $r$ and $m \geq 2$ be  integers, then:
\begin{enumerate}[(a)]
\item Given any sets of integers $\{u_1, \ldots, u_m \}$ and $\{s_1, \ldots, s_m \}$, there exist integers $\{c_1, \ldots, c_m \}$ such that $\sum_i c_i s_i=0$ and ${\rm{gcd}}(u_1-c_1, \ldots, u_m-c_m )=1$.
\item If ${\rm{gcd}}(r,s_1, \ldots, s_m)=1$, then there are integers $\{h_1, \ldots, h_m \}$ such that  \[{\rm{gcd}}(r h_1+s_1, \ldots, rh_m+s_m )=1.\]
\end{enumerate}
\end{lemma}

\begin{proof}

Let $s:=\textrm{gcd}(s_1, \ldots, s_m )$ and consider the surjective map \[f: \mathbb{Z}^m \rightarrow s \mathbb{Z}, (c_1, \ldots, c_m) \mapsto \sum_i c_i s_i\]
since $\mathbb{Z}$ is a PID, there is a exact sequence of $\mathbb{Z}$-modules $\mathbb{Z}^m \xrightarrow{g} \mathbb{Z}^m \xrightarrow{f} s\mathbb{Z} \rightarrow 0  $.\\

For a prime $p$ denote  $v \mapsto \overline{v} $ the canonical projection  $\mathbb{Z} \rightarrow \mathbb{F}_p$ and $\overline{g}: \mathbb{F}_p^m \rightarrow \mathbb{F}_p^m $ the map induced by $g$. We claim that  $\textrm{ker}(\overline{g}) \neq  \mathbb{F}_p^m$. Indeed, tensoring with $\rule{0.3cm}{0.15mm}\otimes_{\mathbb{Z}}\mathbb{F}_p$, we get an exact sequence of $\mathbb{F}_p$-spaces
\[\mathbb{F}_p^m \xrightarrow{\overline{g}} \mathbb{F}_p^m \xrightarrow{\overline{f}} \overline{s} \,\mathbb{F}_p\rightarrow 0  \]
thus if $\textrm{Im}(\overline{g})=0$ we would have $\mathbb{F}_p^{m} \cong \overline{s}\,\mathbb{F}_p$, which contradicts $m \geq 2$.\\

Now set $N:=\sum_i u_i s_i $ and $u:=(u_1, \ldots, u_m)$. First suppose $N\neq 0$ and for each prime $p$ dividing $N$ consider the set 
\[X_p:=\{ \overline{v} \in \mathbb{F}_p^m: \overline{g}(\overline{v}) =\overline{u} \} \]
then either $X_p=\emptyset$ or $X_p=\overline{v_0}+\textrm{ker}(\overline{g})$ with $\overline{v_0} \in X_p$ and from the above paragraph  follows that $X_p\neq \mathbb{F}_p^m$, so we may pick $ \overline{v}_p \in \mathbb{F}_p^m$ such that $\overline{g}(\overline{v_p}) \neq \overline{u}$. By the Chinese remainder theorem we can choose $v \in \mathbb{Z}^m$ such that
\[ v\equiv v_p \mod p \textnormal{ for all }\, p \mid N \]

(if $N=\pm 1$ any $v\in \mathbb{Z}^{m}$ will satisfy this condition). Define $c=(c_1, \ldots, c_m):=g(v) \in \textrm{ker}(f)$, then  $\sum_i c_i s_i=0$  and the integers $\{u_i-c_i\}$ are coprime, otherwise there would be a prime $p$ such that $p \mid u_i-c_i$ for all $i\leq m$,  so $p \mid \sum_i (u_i-c_i) s_i=N$  and we would get
\[ u \equiv c = g(v) \equiv g(v_p) \not\equiv u \mod p \]

a contradiction. This proves $(a)$ whenever $N\neq0$. On the other hand, if $N=0$  pick any non-zero element $(d_1,\ldots,d_m)\in \textrm{Ker}(f)$ (which exists because $m\geq2$) and take $c_i:=u_i-d_i/d$, where $d=\textrm{gcd}(d_1,\ldots,d_m)$. Then, $\sum_{i}c_{i}s_i=N-\sum_{i}d_is_i/d=0$ and $\textrm{gcd}(u_1-c_1,\ldots,u_m-c_m)=\textrm{gcd}(d_1/d,\ldots,d_m/d)=1$, so $(a)$ also holds in this case.\\

To prove  $(b)$, write $1=ar+bs$ with  $a,b \in \mathbb{Z}$ and write $bs=\sum_i u_i s_i$ for some integers $\{ u_i\}$,  by  part $(a)$ exists $\{c_i\}$ such that $\sum_i c_i s_i=0$ and the integers $ v_i:=u_i-c_i$ are coprime. Let $\{h_i\}$ be such that $\sum_i v_i h_i=a$ then

\[ \sum_i v_i (rh_i+s_i)=r \sum_i v_i h_i+ \sum_i v_i s_i=ra+bs=1 \]

and therefore the integers $\{rh_i+s_i \}$ are coprime.

\end{proof}

\begin{proof}[Proof of Proposition \ref{BasisZero}] Let $\{1,\beta_1,\ldots,\beta_{n-1}\}$ be a basis of $O_K$ and let $s_i:=\tr_{K/\mathbb{Q}}(\beta_{i})$ for $ 1 \leq  i \leq n-1$. By hypothesis $(1/k) \tr_{K/\mathbb{Q}}(O_K)=\langle n/k,s_1/k, \ldots, s_{n-1}/k\rangle_{\mathbb{Z}} = \mathbb{Z}$, so applying part $(b)$ of the above lemma  to $m:=n-1 \geq 2$ and $r:=n/k$,  we find $\{h_1,\ldots,h_{n-1} \}$ such that the integers $\{r h_1+s_1/k,\ldots, r h_{n-1}+s_{n-1}/k \}$ are coprime and therefore the last column of some matrix $A$ in $\textrm{GL}_{n-1}(\mathbb{Z})$, that is,
\[(rh_1+s_1/k, \ldots,rh_{n-1}+s_{n-1}/k)^t=A \,(0, \ldots,0,1)^t\]

now define the basis $\{1,\alpha_1,\ldots,\alpha_{n-1}\}$ by the relation

\[ \hspace{4cm} (1,\alpha_1,\ldots,\alpha_{n-1})^t:=\begin{bmatrix}
   1 & 0\\\
   0 & A^{-1}
 \end{bmatrix}  (1,\beta_1,\ldots,\beta_{n-1})^t \] 
then  $(\alpha_1,\ldots,\alpha_{n-1})^t=A^{-1}\, (\beta_1,\ldots,\beta_{n-1})^t$ and 
\begin{equation*}
\begin{split}
(1/k)\left( \tr_{K/\mathbb{Q}}(\alpha_1), \ldots,\tr_{K/\mathbb{Q}}(\alpha_{n-1}) \right)^t &=A^{-1} (s_1/k, \ldots,s_{n-1}/k)^t\\
& \equiv A^{-1} (rh_1+s_1/k, \ldots,rh_{n-1}+s_{n-1}/k)^t \mod r \\
&= (0, \ldots,0, 1)^t  \mod r
\end{split}
\end{equation*}
hence $\left( \tr_{K/\mathbb{Q}}(\alpha_1), \ldots,\tr_{K/\mathbb{Q}}(\alpha_{n-1}) \right) \equiv (0,\ldots,0,k) \, \textrm{mod} \, n$.  \\

To prove that   $\{\alpha_1-t_1 /n, \ldots, \alpha_{n-2}-t_{n-2}/n,(n/k)\alpha_{n-1}-t_{n-1}/k\}$ is a $\mathbb{Z}$-basis of $O_K^0$ note that its $\mathbb{Z}$-span  $C$  is clearly contained in $O_K^0$ and the $\mathbb{Z}$-span of  $\{1,\alpha_1-t_1 /n, \ldots, \alpha_{n-2}-t_{n-2}/n,(n/k)\alpha_{n-1}-t_{n-1}/k\}$ has index $n/k$. Indeed, the matrix expressing  $\{1,\alpha_1-t_1 /n, \ldots, \alpha_{n-2}-t_{n-2}/n,(n/k)\alpha_{n-1}-t_{n-1}/k\}$ in terms of the integral basis $\{1,\alpha_1,\ldots,\alpha_{n-1}\}$ of $K$ is \[\begin{bmatrix}
1&0&0& \ldots &0\\
-t_1/n & 1 & 0 & \ldots &0\\
-t_2/n & 0 & 1 & \ldots &0\\
\vdots & \vdots & \vdots & \ddots & \vdots\\
-t_{n-1}/k & 0 &0 & \ldots &n/k
\end{bmatrix}\]

so its determinant is $n/k$. It follows that  $\d(C)=\frac{n}{k^2} \,\d(K)=\d(O_K^0)$ and therefore $C=O_K^0.$   

\end{proof}

\subsection{Proofs of the main results}

We are now ready to prove the following partial reciprocal of \cite[Lemma 5.1]{Casimir}.

\begin{theorem}\label{MainThLift} Let $K$ be a number field of degree $n$, suppose  that $\tr_{K/\mathbb{Q}}(O_K) = \mathbb{Z}$ and $ \left(\mathbb{Z}/n \mathbb{Z}\right)^\times$ is cyclic. Then given a number field $L$ we have that:
\begin{enumerate}[(i)]
\item Every isometry $\varphi: (O_K^{\bot},\tr_{K/\mathbb{Q}}) \xrightarrow{\sim} (O_L^{\bot},\tr_{L/\mathbb{Q}})$ extends to an isometry \[\phi: (O_K,\tr_{K/\mathbb{Q}}) \xrightarrow{\sim} (O_L,\tr_{L/\mathbb{Q}}).\]
\item If $\left(n,\d(K)\right)=1$, then every isometry $\varphi: (O_K^0,\tr_{K/\mathbb{Q}})  \xrightarrow{\sim}  (O_L^0,\tr_{L/\mathbb{Q}})$ also extends to an isometry $\phi: (O_K,\tr_{K/\mathbb{Q}}) \xrightarrow{\sim} (O_L,\tr_{L/\mathbb{Q}}).$
\end{enumerate}
\end{theorem}
\vspace*{0.3cm}

\begin{proof}
Case $n=1$ is trivial and case $n=2$ is easy to check, so let us suppose that $n \geq 3$.
\begin{enumerate}[(i)]
\item Let $\{1,\alpha_1, \ldots, \alpha_{n-1}\}$ be any basis of $O_K$. Take $t_i:=\tr_{K/\mathbb{Q}}(\alpha_i)$, $1 \leq i < n$, so that $\alpha_{i\bot}=n\alpha_{i}-t_i$ is a basis of $O_K^{\bot}$ and for each  $1 \leq i < n$  choose $\beta_i \in O_L$ and $s_i \in \mathbb{Z}$  such that 
\[ y_i:=\varphi(\alpha_{i\bot})=n \beta_i-s_i \in O_L^{\bot}\]

then $\tr_{L/\mathbb{Q}}(y_i y_j)=n^2\tr_{L/\mathbb{Q}}(\beta_i \beta_j)-ns_is_j \equiv -n s_i s_j \mod n^2$ and since $\tr_{L/\mathbb{Q}}(y_i y_j)=\tr_{K/\mathbb{Q}}(\alpha_{i\bot} \alpha_{j\bot}) \equiv -n t_i t_j \mod n^2$ we conclude \[ t_it_j \equiv s_i s_j \mod n \textnormal{, for all} \, 1 \leq i,j<n \]

Let $\{ u_i\}$ be integers such that $\sum_i u_i t_i \equiv 1 \mod n$ and take $u:=\sum_i u_i s_i$, then $ s_i u \equiv t_i \mod n$ for all $1\leq i< n $,
in particular $(u,n)=1$ and if $v$ is its inverse modulo $n$ then
\[t_i t_j \equiv  v^2  t_i t_j \mod n \textnormal{, for all} \, 1 \leq i,j<n \]

thus $v^2 \equiv 1 \mod n $ and as  $ \left(\mathbb{Z}/n \mathbb{Z}\right)^\times$ is cyclic this implies $v \equiv \pm 1 \mod n$, thus $ s_i  \equiv \pm t_i \mod n \, \textnormal{ for all} \, 1 \leq i<n$ and by Lemma $\ref{Lifting}$ either $\phi=\varphi^+$ or $\phi=\varphi^-$ extend $\varphi$ to an isometry $(O_K,\tr_{K/\mathbb{Q}}) \cong (O_L,\tr_{L/\mathbb{Q}}).$

\item First notice that we must have $\tr_{L/\mathbb{Q}}(O_L)=\mathbb{Z}$. If $l$ is a positive integer such that $\tr_{L/\mathbb{Q}}(O_L)=l\Z$  we have, thanks to Lemma \ref{DetsRelativos}, that the existence of the isometry $\varphi$ implies $n \, \d(K)=\frac{n}{l^2} \, \d(L)$. Notice that $\tr_{L/\mathbb{Q}}(O_L)=l \,\mathbb{Z}$ implies that $l^{n} \mid \d(L)$ and $l \mid n$. In particular, $ l^{n-2} \mid  \left(l^{-2} \d(L)\right)$ i.e., $ l^{n-2} \mid \d(K)$. Since $l \mid n$ and $\left(  \d(K),n\right)=1$, we conclude that $l=1$.

Now, take a basis $\{1,\alpha_1, \ldots, \alpha_{n-1}\}$ of $O_K$  and $\{t_1,\ldots,t_{n-1}\}$ as in Proposition \ref{BasisZero} so that \[w_1:=\alpha_1-t_1/n,\ldots,w_{n-2}:=\alpha_{n-2}-t_{n-2}/n,w_{n-1}:=n\alpha_{n-1}-t_{n-1}\] is a basis of $O_K^0$. Note that for all $1\leq i<n$ we have 
\begin{equation}\label{eqn:ngood}
\tr_{K/\mathbb{Q}}(w_i w_{n-1})=\tr_{K/\mathbb{Q}}\left(w_i(n \alpha_{n-1}-t_{n-1})\right)=n\tr_{K/\mathbb{Q}}(w_i\alpha_{n-1}) \equiv 0 \mod n\tag{$*$} 
\end{equation}

 and for $i=n-1$
\[  (1/n)\tr_{K/\mathbb{Q}}( w_{n-1}^2) =\tr_{K/\mathbb{Q}}(n\alpha_{n-1}^2-t_{n-1}\alpha_{n-1})=n\tr_{K/\mathbb{Q}}(\alpha_{n-1}^2)-t_{n-1}^2 \equiv -1 \mod n\]

Let us define $\theta_i:=\varphi(w_i)$, $1\leq i<n$. Take $\gamma_L \in O_L$ such that $\tr_{K/\mathbb{Q}}(\gamma_L)=1$ and write \[ 1-\gamma_L \, n =\sum_{i=1}^{n-1} l_i \theta_i \textnormal{ with } l_i \in \mathbb{Z}\]

taking traces in the congruences $\theta_j \equiv \sum_{i}l_i \theta_i \theta_j \mod n \,$ we find that the Gram matrix $G:=(\tr_{K/\mathbb{Q}}(w_i w_j))=(\tr_{L/\mathbb{Q}}(\theta_i \theta_j) )$ satisfies \[G \, (l_1,\ldots,l_{n-1})^t \equiv 0 \mod n\]
Call $G^{*}$ the matrix obtained from $G$  dividing the last column by $n$, which has integer entries thanks to  \eqref{eqn:ngood}, then \[G^{*} \, (l_1,\ldots,l_{n-2},0)^t \equiv G^{*} \, (l_1,\ldots,l_{n-2,}nl_{n-1})^t=G \, (l_1,\ldots,l_{n-2},l_{n-1})^t \equiv 0 \mod n.\]

From Lemma \ref{DetsRelativos} we know that $\textrm{det}(G^{*})=\d(K)$ which by hypothesis is coprime to $n$. It follows that $l_i \equiv 0 \mod n $ for all $1\leq i \leq n-2$ and thus $l_{n-1} \theta_{n-1}-1 \equiv 0 \mod n$, squaring this congruence and taking traces again we find
\begin{align*}
l_{n-1}^2 \theta_{n-1}^2-2l_{n-1} \theta_{n-1}+1 \equiv 0 \mod n^2 & \Rightarrow l_{n-1}^2 \tr_{L/\mathbb{Q}}(\theta_{n-1}^2)+n \equiv 0 &\mod n^2\\
 & \Rightarrow l_{n-1}^2 \tr_{K/\mathbb{Q}}(w_{n-1}^2)+n \equiv 0 &\mod n^2\\
  & \Rightarrow l_{n-1}^2 (-n)+n \equiv 0 &\mod n^2\\
  & \Rightarrow l_{n-1}^2 (-1)+1 \equiv 0 &\mod n\,\,\\
  & \Rightarrow l_{n-1}^2 \equiv 1 &\mod n\,\,
\end{align*}

Since $\left(\mathbb{Z}/n \mathbb{Z}\right)^\times$ is cyclic, this implies $l_{n-1} \equiv \pm 1 \mod n$, now let \[\gamma_K:=\alpha_{n-1}-(t_{n-1}-1)/n=(1/n)(w_{n-1}+1) \in O_K\] then, $\tr_{K/\mathbb{Q}}(\gamma_K)=1$ and $\varphi(1-n\gamma_K)=-\theta_{n-1} \equiv \mp 1 \mod n$, thus we are done by  Lemma $\ref{Lifting}(ii)$. 
\end{enumerate}
\end{proof}

\begin{corollary}\label{DiscFundEquiv}
 Let $K$ be totally real number field of fundamental discriminant and degree $n\ge 3$. If  $ \left(\mathbb{Z}/n \mathbb{Z}\right)^\times$ cyclic, then for any number field $L$ the following are equivalent:
\begin{enumerate}[(i)]
\item $(O_K,\tr_{K/\mathbb{Q}})  \cong (O_L,\tr_{L/\mathbb{Q}})$.

\item $(O_K^{\bot},\tr_{K/\mathbb{Q}})  \cong (O_L^{\bot},\tr_{L/\mathbb{Q}})$.

\item ${\rm Sh}(K)={\rm Sh}(L)$ and $L$ is totally real with fundamental discriminant.
\end{enumerate}
If $\left(n,\d(K)\right)=1$ then the three items are also equivalent to (iv) $(O_K^{0},\tr_{K/\mathbb{Q}})  \cong (O_L^{0},\tr_{L/\mathbb{Q}})$.
\end{corollary}

\begin{proof}
The equivalences $(i) \iff (ii)$ and $(i) \iff (ii) \iff (iv)$ under the additional conditions follow from \cite[Lemma 5.1]{Casimir}, Theorem $\ref{MainThLift}$ and the fact that $\tr_{K/\mathbb{Q}}(O_K)=\mathbb{Z}$ (since $K$ has fundamental discriminant). It remains to prove $(ii) \iff (iii)$: \\

$\Rightarrow)$ Since $L$ is totally real  $b_K=\tr_{K/\mathbb{Q}}$ and ${\rm Sh}(K)={\rm Sh}(L)$. Furthermore, $\d(L)=\d(K)$ is fundamental.\\
$\Leftarrow)$  Let $\lambda \in \mathbb{R}^\times$ be such that  an isometry $\varphi:(O_K^{\bot},\tr_{K/\mathbb{Q}}) \rightarrow (O_L^{\bot}, \lambda \, \tr_{L/\mathbb{Q}}) $ exists. Let $\{1,\alpha_1,\ldots,\alpha_{n-1} \}$ be a basis of $O_K$, take  $x_i:=\alpha_{i\bot}=n\alpha_i-t_i$, $1\leq i<n$ $t_i:=\tr_{K/\mathbb{Q}}(\alpha_i)$, as basis of $O_K^{\bot}$ and define $y_i:=\varphi(x_i)=n\beta_i-s_i$, $\beta_i \in O_L$, $s_i\in \mathbb{Z}$. Then, $\tr_{K/\mathbb{Q}}(x_ix_j)= \lambda \, \tr_{L/\mathbb{Q}}(y_iy_j) $ for all $1\leq i,j <n$, in particular is a positive rational so we may write $\lambda=r/s$ where $r$ and $s$ are coprime positive integers.\\

Since $n^{2n-3}\d(K)=\d(O_K^{\bot})=\det(O_L^{\bot},\lambda \, \tr_{L/\mathbb{Q}})=\lambda^{n-1}n^{2n-3} \d(L)$ we know that \[ s^{n-1}\d(K)=r^{n-1}\d(L)\]

Thus $r^2 \mid \d(K)$ and if we suppose $r \neq 1$ then $r=2^u$ with \[n-1 \leq (n-1)u\leq v_2(\d(K)) \leq 3\]
also $s^2 \mid \d(L)$ and $(r,s)=1$ forces $s=1$. Therefore $\d(K)=2^{u(n-1)} \d(L)$ and:
\begin{itemize}
\item If $n=4$, we would have $u=1$ and $\lambda=2$, but  since $ (1/4) \tr_{K/\mathbb{Q}}(x_ix_j) \equiv -t_it_j \mod 4$ and $(1/4)\tr_{L/\mathbb{Q}}(y_iy_j) \equiv -s_is_j \mod 4$, this implies
\[t_i t_j\equiv -(1/4) \tr_{K/\mathbb{Q}}(x_ix_j)=-(1/2) \tr_{K/\mathbb{Q}}(y_iy_j) \equiv 2s_is_j \mod 4 \]
which yields $t_it_j \equiv 0 \mod 2$ for all $1\leq i,j<n $ contradicting $\tr_{K/\mathbb{Q}}(O_K)=\mathbb{Z}$.
 \item If $n=3$, we would have $u=1$, so either $v_{2}(\d(K))=2$ and $v_{2}(\d(L))=0$ in which case $\d(L)=\d(K)/4 \equiv 3 \mod 4$ or  $v_{2}(\d(K))=3$ and $v_{2}(\d(L))=1$. Both cases lead to a contradiction since by Stickelberger's criterion, applied to $\d(L)$, the discriminant of a number field is equal to either $0$ or $1$ modulo $4$.
\end{itemize}
Thus $r=1$ and by symmetry of the argument $s=1$, therefore $\lambda=1$ and $(O_K^{\bot},\tr_{K/\mathbb{Q}})  \cong (O_L^{\bot},\tr_{L/\mathbb{Q}})$, as required.

\end{proof}

\begin{theorem}\label{Main}
 Let $K$ be totally real number field of fundamental discriminant and degree $n\ge 3$. If  $ \left(\mathbb{Z}/n \mathbb{Z}\right)^\times$  is cyclic, then for any number field $L$ the following are equivalent:
\begin{enumerate}[(i)]
\item $K \cong L$.

\item $(O_K^{\bot},\tr_{K/\mathbb{Q}})  \cong (O_L^{\bot},\tr_{L/\mathbb{Q}})$.

\item ${\rm Sh}(K)={\rm Sh}(L)$ and $L$ is totally real with fundamental discriminant.
\end{enumerate}

If $\left(n,\d(K)\right)=1$, then the three items are also equivalent to (iv) $(O_K^{0},\tr_{K/\mathbb{Q}})  \cong (O_L^{0},\tr_{L/\mathbb{Q}})$.

\end{theorem}

\begin{proof}
The result follows from Corollary \ref{DiscFundEquiv} and \cite[Theorem 5.3]{Casimir}.
\end{proof}

\subsection{Strategy's limitations}

The following examples try to illustrate the limitations of the strategy employed to prove Theorems \ref{MainThLift}, \ref{Main} and  test the sharpness of the statements. All the examples here have been found by conviniently looking at John Jones tables \cite{Jones}. The calculations of sizes of orthogonal groups have been carried out with \texttt{Magma}.\\

Suppose a number field $K$  is totally real with $\tr_{K/\mathbb{Q}}(O_K)=k\mathbb{Z}$, $k\in \mathbb{Z}^+$, then by \cite[Lemma 5.1]{Casimir} the restrictions maps \[\textrm{Aut}(O_K,\tr_{K/\mathbb{Q}}) \rightarrow \textrm{Aut}(O_K^{\bot},\tr_{K/\mathbb{Q}}) \textnormal{ and \,} \textrm{Aut}(O_K,\tr_{K/\mathbb{Q}}) \rightarrow \textrm{Aut}(O_K^{0},\tr_{K/\mathbb{Q}}) \]

\noindent are well defined homomorphism of groups. Moreover, this maps are injective if $\displaystyle{\frac{n}{k}}\nmid 2$. This is  because by \cite[Lemma 5.1]{Casimir} given $\varphi$ in either codomain, $\varphi^+$ and $\varphi^-$ are the only possible pre-images of $\varphi$ and they cannot both extend $\varphi$, otherwise, we would have  \[n \varphi^+(\gamma_K)-k= \varphi^+(n\gamma_K-k)=\varphi^-(n\gamma_K-k)=n \varphi^-(\gamma_K)+k,\] where $\gamma_K\in O_K \textnormal{ and } \tr_{K/\mathbb{Q}}(\gamma_K)=k$
so $-k \equiv k \mod n $ and $\displaystyle{\frac{n}{k}}\mid 2$. It follows that, when $\displaystyle{\frac{n}{k}}\nmid 2$, every automorphism of $(O_K^{\bot},\tr_{K/\mathbb{Q}})$ can be extended to all $O_K$ if and only if $\#\textrm{Aut}(O_K,\tr_{K/\mathbb{Q}})= \# \textrm{Aut}(O_K^{\bot},\tr_{K/\mathbb{Q}})$ and in general the restriction map will not be surjective if $\#\textrm{Aut}(O_K,\tr_{K/\mathbb{Q}})<\# \textrm{Aut}(O_K^{\bot},\tr_{K/\mathbb{Q}})$, similarly for $O_K^0$.

\begin{example}
Let $K$ be the cubic field with defining polynomial $x^3+x^2-8x+3$, then $K$ is totally real, $\d(K)=3\cdot 5^2\cdot19$, $\tr_{K/\mathbb{Q}}(O_K)=\mathbb{Z}$ and 
\[ \#\textrm{Aut}(O_K^{\bot},\tr_{K/\mathbb{Q}})= \#\textrm{Aut}(O_K,\tr_{K/\mathbb{Q}})=2 \neq 4=\#\textrm{Aut}(O_K^0,\tr_{K/\mathbb{Q}})\]
hence there exists an automorphism of $(O_K^0,\tr_{K/\mathbb{Q}})$ that \emph{cannot} be extended to all $O_K$, so the hypothesis $(n,\d(K))=1$ in Theorem $\ref{MainThLift}(ii)$ cannot be dropped. By contrast observe that all automorphisms of  $(O_K^\bot,\tr_{K/\mathbb{Q}})$ can be extended to  $O_K$.
\end{example}

\begin{example} Let $K$ be the field with defining polynomial $x^4-2x^3-5x^2+6x+1$, then
$K$ is totally real, $\d(K)=2^6\cdot5^2$,\, $\tr_{K/\mathbb{Q}}(O_K)=2 \,\mathbb{Z}$ and \[ \#\textrm{Aut}(O_K,\tr_{K/\mathbb{Q}})=8<16=\#\textrm{Aut}(O_K^{\bot},\tr_{K/\mathbb{Q}}).\]

Thus the restriction map is not surjective and therefore the hypothesis $\tr_{K/\mathbb{Q}}(O_K)=\mathbb{Z}$ in Theorem $\ref{MainThLift}(i)$ is not superfluous. \\
\end{example}

\begin{proposition}\label{MyM} Let $K$ be a $\mathbb{Z}/l\mathbb{Z}$-field with $l$ prime, then  for any number field $L$, every isometry $(O_K^\bot,\tr_{K/\mathbb{Q}}) \cong (O_L^\bot,\tr_{L/\mathbb{Q}})$ (resp.  $(O_K^0,\tr_{K/\mathbb{Q}}) \cong (O_L^0,\tr_{L/\mathbb{Q}})$) can be extended to one between the integral trace quadratic modules $(O_K,\tr_{K/\mathbb{Q}}) \cong (O_L,\tr_{L/\mathbb{Q}})$.
\end{proposition}
\begin{proof}
In the case $l\nmid \d(K)$, this follows  directly from Theorem $ \ref{MainThLift}$ and Corollary $\ref{TraceZ}(i)$. On the other hand, if $l\mid \d(K)$, we have that $l$ ramifies in $K$ and the fact that $K$ is a  $\mathbb{Z}/l\mathbb{Z}$-field  forces to be only one prime $\mathfrak{p}$ in $K$ lying over $l$ and for that prime $l=e(\mathfrak{p}|l)$, so $(\ref{TraceMau})$ tells us $\tr_{K/\mathbb{Q}}(O_K)=l \, \mathbb{Z}$. We claim that if $(O_K^0,\tr_{K/\mathbb{Q}}) \cong (O_L^0,\tr_{L/\mathbb{Q}})$ this implies $\tr_{L/\mathbb{Q}}(O_L)=l \, \mathbb{Z}$.\\

Suppose not, then  $\tr_{L/\mathbb{Q}}(O_L)=\mathbb{Z}$ and by Lemma \ref{DetsRelativos} we would have $ \d(L)=l^{-2}\,\d(K)$. If $l=2$, then $L=\Q(\sqrt{\d(L)})=\Q(\sqrt{l^{-2}\d(K)})=K$ and since $\tr_{K/\mathbb{Q}}(O_K)=l \, \mathbb{Z}$ we would obtain a contradiction, so we may assume $l$ odd. Recall from  Corollary $\ref{TraceZ}(i)$ that $l\leq v_l(\d(K))$, as $l$ is odd and $K$ is a $\mathbb{Z}/l\mathbb{Z}$-field, we know that $\d(K)$ is perfect square, so in fact $l+1\leq v_l(\d(K))$, that is, $l-1\leq v_l(\d(L))$ but by Corollary $\ref{TraceZ}(ii)$ $l$ is tamely ramified in $L$,  thus if $f_1^{e_1} \ldots f_g^{e_g}  $ is its factorization type, the inequality \[l-1 \leq  v_l(\d(L))=l-(f_1+\cdots +f_g)\leq l-1\] shows $f_1=1$ and $g=1$ and this yields $l=e_1f_1=e_1$, a contradiction.\\

\noindent Hence if $\varphi:  (O_K^0,\tr_{K/\mathbb{Q}}) \xrightarrow{ \sim} (O_L^0,\tr_{L/\mathbb{Q}})$ is an isometry,  we have $\tr_{K/\mathbb{Q}}(O_K)=l \, \mathbb{Z}=\tr_{L/\mathbb{Q}}(O_L)$ and taking $\gamma_K=1$ in $(\ref{Lifting})(ii)$ we conclude that both $\varphi^+$ and $\varphi^-$ extend $\varphi$. Also if we are given an isometry $\varphi:  (O_K^\bot,\tr_{K/\mathbb{Q}}) \xrightarrow{ \sim} (O_L^\bot,\tr_{L/\mathbb{Q}})$, then $\d(K)=\d(L)$ and, as $l$ is prime,  Corollary $\ref{TraceZ}$ implies $\tr_{L/\mathbb{Q}}(O_L)=l \, \mathbb{Z}$, thus $(\ref{BasisZero})$ shows $O_K^\bot=l \,O_K^0$ and $O_L^\bot=l\,O_L^0$, therefore $\varphi(O_K^0)=O_L^0$ ($\varphi$ considered from $K^0$ to $L^0$) and we are back to the above case. \end{proof}

\begin{example} Let $K$ be the sextic field with defining polynomial $x^6-x^5-6x^4+6x^3+8x^2-8x+1$, then $K$ is a totally real $\mathbb{Z}/6\mathbb{Z}$-field,
$\d(K)=3^3\cdot7^5$,\, $\tr_{K/\mathbb{Q}}(O_K)=\mathbb{Z}$ and 
\[ \#\textrm{Aut}(O_K^\bot,\tr_{K/\mathbb{Q}})= \#\textrm{Aut}(O_K,\tr_{K/\mathbb{Q}})=96 \neq 1440=\#\textrm{Aut}(O_K^0,\tr_{K/\mathbb{Q}})\]

\end{example}

As a side note, we did not find any example showing that $\left(\mathbb{Z}/n \mathbb{Z}\right)^\times$ being cyclic is really a necessary hypothesis, so there might be some place for improvement there. In particular, it would be interesting to answer the following question.

\begin{question}\label{Question8}
Do there exist  totally real octic fields with odd discriminant $K,L$, and an isometry $ (\mathcal{O}_K^{0},\textrm{tr}_{K/\mathbb{Q}})\cong (\mathcal{O}_L^{0},\textrm{tr}_{L/\mathbb{Q}})$ which cannot be lifted to an isometry $(\mathcal{O}_K,\textrm{tr}_{K/\mathbb{Q}})\cong (\mathcal{O}_L,\textrm{tr}_{L/\mathbb{Q}})$?
\end{question}

\section{Casimir Invariants}\label{CasimirInv}

In this section we recall the definition of Casimir invariants and state some of the useful facts about them. For a detailed treatment see \cite{Casimir}.\\

Let $V$ be a finite dimensional vector space over a field $F$ and let $V^{*}:={\rm Hom}_F(V,F)$ be its dual. Let $\Gamma: V \rightarrow V^{*}$ be an isomorphism, for instance the one induced by a nondegenerate bilinear form $B:V\times V\rightarrow F$.  Now let $R$ be an $F$-algebra and let $\phi, \psi \in {\rm Hom}_{F}(V,R)$. The map \[V^{*} \times V\rightarrow R,\; (f,v) \mapsto (\psi\circ\Gamma^{-1})(f)\cdot \phi(v)\]
is bilinear and $F$-balanced (as $R$ is an $F$-algebra), hence it lifts to a morphism $V^{*} \otimes_F V \rightarrow R$. Identifying $V^{*} \otimes_F V$ with $\text{End}_{F}(V)$ we obtain a linear map
 \[\rho_{\Gamma,\psi,\phi}: \text{End}_{F}(V) \rightarrow R.\] 

\begin{definition}
Let $\psi, \phi \in {\rm Hom}_{F}(V,R)$. The $\Gamma$-Casimir element of $\psi$ and $\phi$  is the element $c_{\Gamma}(\psi, \phi) \in R$ given by the image under $\rho_{\Gamma,\psi,\phi}$ of the identity morphism; \[c_{\Gamma}(\psi, \phi) := \rho_{\Gamma,\psi,\phi}(1).\]  
\end{definition}

\begin{definition}
Let $F$ be a field, $V$ be a finite dimension $F$-space, $\Gamma :V \to V^{*}$ be an isomorphism and $R$ be an $F$-algebra. The \textbf{Casimir pairing} associated to $\Gamma$ is the  map
\[\langle \cdot , \cdot\rangle_{\Gamma} :{\rm  Hom}_{F}(V,R) \times {\rm Hom}_{F}(V,R) \rightarrow R,\; (\psi,\phi) \mapsto c_{\Gamma}(\psi, \phi).\]
\end{definition}

Whenever the isomorphism $\Gamma$ is induced by a nondegenerate bilinear form $B$ we denote by $\langle \cdot , \cdot\rangle_{B}$ the Casimir pairing associated to $\Gamma$.\\

If $K/F$ is a finite separable field extension the trace pairing \[\tr_{K/F}: K\times K \to F; \ \   ( x, y ) \mapsto \tr_{K/F}(xy)\] is a nondegenerate bilinear form. For any $F$-algebra $R$ we denote by  $\langle \cdot, \cdot\rangle_{\tr_{K/F}}$ the Casimir pairing on ${\rm Hom}_{F}(K,R)$ associated to the trace pairing ${\rm tr}_{K/F}$.

\begin{theorem}\cite[Corollary  2.9]{Casimir}
Let $F$ be a field and let $K/F$ and $L/F$ be  separable field extensions of the same degree $n$. Suppose that $\Omega/F$ is field extension containing a Galois closure of $KL/F$. Let $\phi: K \rightarrow L$ be an $F$-linear map. Then the following are equivalent
\begin{enumerate}[(i)]
\item The map $\phi$ is an isometry between $(K,{\rm tr}_{K/F})$ and $(L,{\rm tr}_{L/F})$.
\item The map $\Phi$, composition by $\phi$, is an isometry between the spaces ${\rm Hom}_F(L,\Omega)$  and ${\rm Hom}_F(K,\Omega)$ endowed with their Casimir pairings $\langle \cdot ,\cdot \rangle_{{\rm tr}_{L/F}}$ and $\langle \cdot ,\cdot \rangle_{{\rm tr}_{K/F}}$ 
\item The matrix $U=\left(c_{ij} \right)$ is orthogonal, where $c_{ij}:=\langle \sigma_i, \tau_j\phi \rangle_{{\rm tr}_{K/F}}$, and  $\{\sigma_i\}$, $\{\tau_i\}$ are the sets of $F$-embeddings of $K$ and $L$ into $\Omega$.
\end{enumerate}
\end{theorem}

\section{A proof via Bhargava's parametrization of quartic rings}\label{QuarticParametrization}

The aim of this section is to prove  \cite[Conjecture 2.10]{Man2} using Bhargava's parametrization of quartic rings. Even though the veracity of the conjecture follows from  Theorem \ref{Main} we add the proof coming from Bhargava's parametrization since it generalizes the ideas involved in the proof of cubic fields, thus showing the close relation existing between parametrization of rings and trace-zero forms on them. A similar, although quite more simple, argument can be carried on for quadratic fields given that there is also a  parametrization in degree $2$. For the convenience of the reader we recall here the statement of the conjecture

\begin{conje}\cite[Conjecture 2.10]{Man2}
Let $K$ be a totally real quartic number field with fundamental discriminant. If $L$
is a tamely ramified number field such that an isomorphism of quadratic modules
\[(O_{K}^{0}, \tr_{K/\Q}) \cong (O_{L}^{0}, \tr_{L/\Q})\] exists, then $K \cong L$.
\end{conje}

Thanks to Lemma \ref{TameYFundaIsSquafree} we may assume in the conjecture that $K$ and $L$ have square free discriminant.

\subsection{Parametrization of quartic rings}

Let $(\mathbb{Z}^2\otimes \textnormal{\textrm{Sym}}^2 \mathbb{Z}^3)^*$ denote the set of pairs $(A,B)$ of \textit{integral} ternary quadratic forms. We can write a pair $(A,B)\in (\mathbb{Z}^2\otimes \textnormal{\textrm{Sym}}^2 \mathbb{Z}^3)^*$ as
$$2\cdot(A,B)=\left(\begin{bmatrix}
    2a_{11}& a_{12} & a_{13}\\
    a_{12}& 2a_{22} & a_{23} \\
    a_{13}& a_{23} & 2a_{33} 
\end{bmatrix},\begin{bmatrix}
    2b_{11}& b_{12} & b_{13}\\
    b_{12} & 2b_{22} & b_{23} \\
    b_{13} & b_{23} & 2b_{33} 
\end{bmatrix}\right)$$
where $a_{ij},b_{ij} \in \mathbb{Z}$. The group $G_{\Z}:=\textrm{GL}_2(\mathbb{Z}) \times  \textrm{SL}_3(\mathbb{Z})$ acts naturally on  $(\mathbb{Z}^2\otimes \textnormal{\textrm{Sym}}^2 \mathbb{Z}^3)^*$. Namely, if $g=(g_2,g_3)\in \textrm{GL}_2(\mathbb{Z}) \times  \textrm{SL}_3(\mathbb{Z})$ with $g_2=\begin{pmatrix}
r & s\\[-1pt]
t & u
\end{pmatrix}$, then $g$ acts on $(A,B)$ by

\begin{equation}\label{action}
g\cdot(A,B)=\left(r\cdot g_3Ag_3^t+s\cdot g_3Bg_3^t,t\cdot g_3Ag_3^t+u\cdot g_3Bg_3^t\right)
\end{equation}

This action has a fundamental invariant called the \textbf{discriminant} and is given by $$\d((A,B))=\d(f_{(A,B)}(x,y))=b^2c^2-27a^2d^2+18abcd-4ac^3-4b^3d$$
where $f_{(A,B)}(x,y)=4\cdot\textrm{det}(Ax-By)=ax^3+bx^2y+cxy^2+dy^3$ is the \textbf{cubic resolvent form} of $(A,B)$, a covariant for the action of $\textrm{GL}_2(\mathbb{Z})$.\\

 In \cite{bhargavaIII} Bhargava proved that we can parametrize quartic rings using integral ternary quadratic forms, his main result is the following.

\begin{theorem}\label{Corres}
There is a bijection between the set of $G_{\Z}$-orbits on the space  $(\mathbb{Z}^2\otimes \textnormal{\textrm{Sym}}^2 \mathbb{Z}^3)^*$ and isomorphism classes of pairs $(Q,R)$, where $Q$ is a quartic ring and $R$ is a cubic resolvent of $Q$. Moreover, this correspondence is discriminant preserving $\textnormal{\d}((A,B))=\textnormal{\d}(Q)=\textnormal{\d}(R)$.
\end{theorem} 

A cubic resolvent of a quartic ring $Q$ is a cubic ring $R$ such  equipped with a
 quadratic resolvent mapping $Q/\Z\rightarrow R/\Z$ satisfying certain formal properties whose precise definition can be found in \cite[Section 3.9]{bhargavaIII}. When $Q$ is the maximal order in a $S_4$-field $K$ the ring $R$ is unique. It is the cubic ring corresponding to $f_{(A,B)}$ by the Delone-Faddeev-Gross parametrization of cubic rings, an order in the usual cubic resolvent field of $K$.\\

The space $(\mathbb{Z}^2\otimes \textnormal{\textrm{Sym}}^2 \mathbb{Z}^3)^*$ has a  $\textrm{SL}_3$-covariant  of degree $4$. Namely, let $(A,B) \in (\mathbb{Z}^2\otimes\textnormal{\textrm{Sym}}^2 \mathbb{Z}^3)^*$ and suppose $Q$ is the quartic ring corresponding to $(A,B)$ by Theorem \ref{Corres}, then the covariant denoted $\mathcal{Q}_{(A,B)}$ is the integral ternary quadratic form obtained by restricting the trace form $\frac{1}{4}\tr(x^2)$  to  $\{x\in \mathbb{Z}+4Q: \tr(x)=0\}$. For example, if $Q_{(A,B)}=O_K$ is the maximal order in a quartic field $K$, then $\mathcal{Q}_{(A,B)}$  corresponds to the isometry class of the quadratic $\mathbb{Z}$-module $(O_K^{\bot},\frac{1}{4}\tr_{K/\mathbb{Q}})$. The explicit computation of $\mathcal{Q}_(A,B)$ in terms of the coefficients of $(A,B)$ is in  Section $2.3.2$ of \cite{piperH} and also in the  appendix  to Chapter $5$ of \cite{bhargavaT}. \\

Thus if $(\textnormal{\textrm{Sym}}^2\mathbb{Z}^3)^*$ denotes the set of integral ternary quadratic forms and \[(\textnormal{\textrm{Sym}}^2\mathbb{Z}^3)^*/{\rm SL}_3(\Z)\] its set of orbits by the action of $\textrm{SL}_3(\mathbb{Z})$, this gives us a map
$$\mathcal{Q}: (\mathbb{Z}^2\otimes \textnormal{\textrm{Sym}}^2 \mathbb{Z}^3)^*/G_{\Z} \longrightarrow (\textnormal{\textrm{Sym}}^2\mathbb{Z}^3)^*/{\rm SL}_3(\Z) .$$
Thanks to Corollary $\ref{DiscFundEquiv}$, the proof of \cite[Conjecture 2.10]{Man2} amounts to proving that $\mathcal{Q}$ is injective when restricted to the orbits of pairs $(A,B)$ coming from totally real quartic fields with square free discriminant.

\subsection{Parametrization of order two ideals in cubic rings}
There is another arithmetic object that is parametrized by pairs of ternary quadratic forms. Let  $\mathbb{Z}^2\otimes \textnormal{\textrm{Sym}}^2 \mathbb{Z}^3$  be the set of pairs $(A,B)$ of symmetric $3\times3$ integer matrices. Again the group $G_{\Z}=\textrm{GL}_2(\mathbb{Z}) \times  \textrm{SL}_3(\mathbb{Z})$ acts naturally on this set as described in the equation (\ref{action}), the only difference is that in the space $\mathbb{Z}^2\otimes \textnormal{\textrm{Sym}}^2 \mathbb{Z}^3$ the \textbf{cubic resolvent form} is now
$$F_{(A,B)}=\textrm{det}(Ax-By)$$
(without the $4$ factor) and the \textbf{discriminant} is defined as $\d((A,B))=\d(F_{(A,B)})$. The following theorem, obtained by Bhargava in \cite{bhargavaII} imposing symmetry on a more general result about $3\times 3\times 2$ boxes of integers (a higher dimensional  analog of Bhargava's cubes), shows how the orbits in this space parametrize order two ideals in cubic rings. 
\begin{theorem}\label{CorrsTrip}
There is a bijection between the set of nondegenerate $G_{\Z}$-orbits on the space $\mathbb{Z}^2\otimes \textnormal{\textrm{Sym}}^2 \mathbb{Z}^3$ and the set of equivalence classes of triples $(R,I,\delta)$, where $R$ is a nondegenerate cubic ring, $I$ is an ideal of $R$, and $\delta$ is an invertible element of $R\otimes \mathbb{Q}$  such that $I^2\subset (\delta)$ and $N(\delta)=N(I)^2$. (Here two triples $(R,I,\delta)$ and $(R',I',\delta')$ are equivalent if there exists an isomorphism $\phi: R \rightarrow R'$ and an element $\kappa\in R'\otimes \mathbb{Q}$ such that $I'=\kappa \phi(I)$ and $\delta'=\kappa^2 \phi(\delta)$). Under this bijection, $\d((A,B))=\d(R)$.
\end{theorem}
The ring $R$ associated to the pair $(A,B)$ is the one corresponding by the Delone-Faddeev-Gross parametrization of cubic rings to $F_{(A,B)}$. Let us denote the correspondence from Theorem \ref{CorrsTrip} as
 $$\Phi: W_{\Z}/G_{\Z}\longrightarrow \mathcal{R}$$
where $W_{\Z}$ denotes the set of nondegenerate elements in  $\mathbb{Z}^2\otimes \textnormal{\textrm{Sym}}^2 \mathbb{Z}^3$.\\
 
Now there is also a natural map 
\[T: \mathcal{R} \longrightarrow (\textrm{Sym}^2 \mathbb{Z}^3)^*/{\rm SL }_3(\Z)\]
taking the equivalence class of $(R,I,\delta)$ to equivalent class of the integral quadratic form obtained by restricting of the trace form $\tr(x^2/\delta)$ to $I$. It is easy to check that $T$ is  discriminant preserving, i.e., $\d(R)=\d\left(I,\tr(x^2/\delta)\right)$.\\

Finally, notice that there is  map connecting the two previous theorems. Indeed, let $V_{\Z}$ denote the set of nondegenerate elements in $(\mathbb{Z}^2\otimes \textnormal{\textrm{Sym}}^2 \mathbb{Z}^3)^*$. Then we have the natural map
\[V_{\Z}/G_{\Z}\rightarrow W_{\Z}/G_{\Z}\]
taking the orbit of $(A,B)\in V_{\Z}$ to the orbit of $(2A,2B) \in W_{\Z}$.

\subsection{Proof of the Conjecture}
All these maps fit together in the following diagram:\\

\hspace*{3cm}
\begin{tikzpicture}[scale=2.4]
\node (B) at (0,2) {$V_{\Z}/G_{\Z}$};
\node (C) at (0,1) {$W_{\Z}/G_{\Z}$};
\node (D) at (0,0) {$\mathcal{R}$};
\node (E) at (2,0) {$(\textrm{Sym}^2 \mathbb{Z}^3)^*/{\rm SL_{3}}(\Z)$};

\path[->,font=\scriptsize,>=angle 90]
(B) edge node [right]{}(C)
(C) edge node [left]{$\Phi$}(D)
(B) edge [bend left] node [right] {$\mathcal{Q}$} (E)
(D) edge node[above]{$T$} (E);

\end{tikzpicture}

\noindent Surprisingly enough we get the following.

\begin{lemma}\label{DiagConm}
The above diagram is commutative.
\end{lemma}
 
 \begin{proof}
 
For the moment we are only able to prove this by a direct computation. Let  $(A,B) \in V_{\Z}$ have coefficients given by
\[2\cdot(A,B)=\left(\begin{bmatrix}
    2a_{11}& a_{12} & a_{13}\\
    a_{12}& 2a_{22} & a_{23} \\
    a_{13}& a_{23} & 2a_{33} 
\end{bmatrix},\begin{bmatrix}
    2b_{11}& b_{12} & b_{13}\\
    b_{12} & 2b_{22} & b_{23} \\
    b_{13} & b_{23} & 2b_{33} 
\end{bmatrix}\right),\]

\noindent and suppose that its class corresponds by Theorem \ref{Corres}  to the class of the pair $(Q,R')$.\\

We need to show that if $(R,I,\delta)$ represents  the class corresponding to  the image of the class of $(2A,2B) \in W_{\Z}$ under $\Phi$, then for some basis of $\langle \alpha_1,\alpha_2,\alpha_3 \rangle$ of $I$ the matrix  $({\rm Tr}(\frac{\alpha_i \alpha_j}{\delta}))$ coincides with the matrix $(Q_{ij})$ in \cite[Pag. 44-45]{piperH}, which is the Gram matrix of $\mathcal{Q}(A,B)$ in an explicit basis of $Q^{\bot}$.\\

The ring $R$ corresponds under the Delone-Faddeev parametrization to the binary cubic form
\[ {\rm det}(2Ax-2By)=ax^3+b x^2 y+ c xy^2+dy^3\]
which means that $R'$ has a basis $\langle 1, \omega, \theta \rangle$  such that the ring structure of $R$ is given by

\begin{align*}
    \omega \theta &= - ad\\
    \omega^2 &=-ac+b \omega -a \theta\\
    \theta^2 &=-bd+d \omega -c \theta \\
\end{align*}
In particular, ${\rm Tr}(\omega)=b$ and ${\rm Tr}(\theta)=-c$.
Now the ideal $I$ has a $\Z$-basis $\langle \alpha_1,\alpha_2,\alpha_3 \rangle$  such that
\[\alpha_i \alpha_j = \delta( c_{ij}+ b_{ij}^* \omega +  a_{ij}^*\theta), \]

where $2A=:(a_{ij}^*)$,  $2B=:(b_{ij}^*)$ and the constants $c_{ij}$ are uniquely determined from the associative law on triple products $(\delta^{-1}\alpha_i )(\delta^{-1}\alpha_j )(\delta^{-1}\alpha_k )$, see \cite[Section 3.1.4]{bhargavaT}. More specifically, these constants are given in terms of $(a_{ij}^*)$ and $(b_{ij}^*)$ by \cite[Equation (3.8)]{bhargavaT}. Alternatively, we can also describe  $C=(c_{ij})$ explicitly as

\[C:={\rm Adj}({\rm Adj}(2A)+{\rm Adj}(2B))- {\rm det}(2A) (2A)-{\rm det}(2B)( 2B) , \]
where ${\rm Adj}(X)$ denotes the adjoint of a matrix $X$. Thus in this basis 
\[ \left( {\rm tr}\left(\frac{\alpha_i \alpha_j}{\delta}\right) \right)=3C+b(2B)-c(2A).\]

\noindent On expanding the right hand side in terms of the $a_{ij}$'s  and $b_{ij}$'s, we find that this matrix coincides with the matrix $(Q_{ij})$ in  \cite[Pag. 44-45]{piperH}.

 \end{proof}

It follows that in order to prove  \cite[Conjecture 2.10]{Man2} all we have to do is prove that $T$ is injective when restricted to the equivalence classes of triples coming from totally real quartic fields of some fixed square free discriminant, say $d$. Denote this subset of $\mathcal{R}$ as $\mathcal{R}(d)$, then the conjecture  follows from. 

\begin{theorem}\label{InjectivityTheorem}
Let $(R,I,\delta)$, $(S,J,\epsilon)$ be triples representing classes in $\mathcal{R}(d)$. If an isomorphism of quadratic modules
$$\left(I,\textnormal{\tr}(x^2/\delta)\right)\cong\left(J,\textnormal{\tr}(x^2/\epsilon)\right) $$
exists, $(R,I,\delta)$ and $(S,J,\epsilon)$ are equivalent.
\end{theorem}

It is convenient to identify some properties of the elements in $\mathcal{R}(d)$, before we give the proof. Start with a pair of integral ternary quadratic forms $(A,B) \in (\mathbb{Z}^2 \otimes \textrm{Sym}^2 \mathbb{Z}^3)^*$ corresponding under $\Psi$ to the pair $(Q,R')$, where $Q$ is the maximal order in a totally real quartic field $K$ of discriminant $d$, and let $(R,I,\delta)$ be a triple corresponding under $\Phi$  with $(2A,2B)$. Then,
\begin{itemize}
\item 
 Since $\d(R')=d$ is square-free, the ring $R'$ is the maximal order in the cubic resolvent field of $K$. Also, $R'$  the cubic ring corresponding to $f_{(A,B)}$ and $R$ is the cubic ring corresponding to $F_{(2A,2B)}=2f_{(A,B)}$; this means that if $\langle 1,\omega,\theta\rangle$ is a normalized $\mathbb{Z}$-basis of $R'$, then $\langle 1,2\omega,2\theta \rangle$ is a normalized basis of $R'$. In particular, $R$ is an order of conductor $\mathfrak{c}=2R'$ in $R'$. Note that given $x=r+2s \omega+2t \theta \in R$, with $r,s,t \in \Z$; we have $x\in 2R'$ if and only if $r \equiv \tr(x)\equiv 0 \mod 2$, also notice that $\tr(x^2) \equiv r^2 \equiv r \equiv \tr(x) \mod 2 $.
\item Since $Q$ is totally real, the pair $(A,B)$ possesses $4$ zeros in $\mathbb{P}^2(\mathbb{R})$ and so does $(2A,2B)$, which means that $\delta$ is totally positive (see \cite[Lemma $21$]{bharVar}).
\item We claim that there is a $\kappa \in R\otimes \mathbb{Q}$, such that $\kappa I$ is an integral ideal in $R$  prime to the conductor $\mathfrak{c}=2R'$. Indeed, take a $\mathbb{Z}$-basis $\{1,\gamma_1,\gamma_2,\gamma_3\}$ of $Q$ and let $t_i:=\tr(\gamma_i)$. Then $\{4\gamma_i-t_i\}$ is a basis of $Q^{\bot}$ and if $(\mathcal{Q}_{ij})$ is the Gram matrix of $\frac{1}{4}\tr(x^2)$ in this basis, then $$\mathcal{Q}_{ii} \equiv t_i \mod 4.$$
Since $d$ is square free, then $(4,t_1,t_2,t_3)=\tr(Q)=\mathbb{Z}$ (see Corollary $\ref{TraceZ}$). Thus at least one of $\mathcal{Q}_{ii}$ must be odd, say $\mathcal{Q}_{11}$. Next, according to Lemma \ref{DiagConm}, $(\mathcal{Q}_{ij})$ is the Gram matrix of $\tr(x^2/\delta)$ in some basis $\{\alpha_1,\alpha_2,\alpha_3\}$ of $I$ and so $\kappa:=\alpha_1/\delta$ is the constant we are looking for. This is because if
$$\frac{\alpha_1^2}{\delta}=r+s(2\omega)+t(2\theta),$$
then $1\equiv\mathcal{Q}_{11} = 3r\equiv r \mod 2$, so $\kappa I\subset \delta^{-1}I^2\subset R$ is an integral ideal such that 
$$1=\frac{\alpha_1^2}{\delta}+\left(1-\frac{\alpha_1^2}{\delta}\right)$$
with $\frac{\alpha_1^2}{\delta} \in \kappa I$ and $1-\frac{\alpha_1^2}{\delta}\in \mathfrak{c}=2R'$.\\

We have proved that $(R,I,\delta)$ is equivalent to a triple $(R,I',\delta')$ where $I'$ is an integral ideal prime to the conductor. This implies, by the same proof given for maximal orders, that if we fix any ideal $\mathfrak{a}$ in $R$ prime to the conductor, then $(R,I,\delta)$  is equivalent to a triple $(R,I'',\delta'')$ where $I''$ is integral prime to the conductor and prime to $\mathfrak{a}$. 
\end{itemize}

\begin{proof}[Proof of Theorem \ref{InjectivityTheorem}] Let $K:=R\otimes\mathbb{Q}$ and $L:=S\otimes \mathbb{Q}$. Choose $I$ and $J$ to be prime the conductor of $R$ and $S$, respectively, and to $d$. This implies, thanks to $I$ and $J$ being invertible, that $I^2=(\delta)$ and $J^2=(\epsilon)$. The isometry can be extended to a rational isometry
$$\phi:\left(K,\textnormal{\tr}(x^2/\delta)\right)\xrightarrow{\sim}\left(L,\textnormal{\tr}(x^2/\epsilon)\right)$$
Let $\sigma:K\hookrightarrow \mathbb{R}$ and $\tau:L \hookrightarrow \mathbb{R}$ be embeddings (recall that  $K$ and $L$ are totally real). We claim that 
$c:=\langle \tau,\sigma\phi \rangle_{\tr_{K/\mathbb{Q}}} \sqrt{\frac{\sigma(\delta)}{\tau(\epsilon)}}$, see \S \ref{CasimirInv}, is an algebraic integer. Similarly to  \cite[Remark 4.4]{Casimir}, we see that it is enough to prove it when $\sigma$ and $\tau$ are some fixed inclusions $\iota_K:K \hookrightarrow \mathbb{R}$ and $\iota_L: L \hookrightarrow \mathbb{R}$ which we use to identify $K$ and $L$ with subsets of $\mathbb{R}$. The strategy is the same as in  \cite[Proposition 4.1]{Casimir}; we prove that 
\begin{equation}\label{pint}
  c^2\otimes1\in O_{KL}\otimes \mathbb{Z}_p, 
\end{equation}

\noindent for all primes $p$. \\

 Let $\langle \alpha_{1}, \alpha_2 ,\alpha_3 \rangle$ be a basis of $I O_K$ and let  $\langle \alpha_{1}^*, \alpha_2^* ,\alpha_3^*\rangle$ be its dual basis in $K$ with respect to the trace form. Then by definition and the fact that Casimir invariants are compatible with tensor products (see \cite[Proposition 2.8(iii)]{Casimir}) we have  
\begin{equation} \label{epdel}
  \frac{\epsilon}{\delta} c^2\otimes 1=\left( \sum_{i=1}^{3}  (\alpha_i^*\otimes 1) (\phi\otimes 1)(\alpha_i\otimes 1) \right)^2.
\end{equation}
Now  $\langle \alpha_{1}^*, \alpha_2^* ,\alpha_3^*\rangle$ generates the dual ideal of $(I O_K)$, i.e, the fractional ideal \[\{x\in K : {\rm Tr}_{K/\mathbb{Q}}(x (IO_K))\subset \Z \}=(I O_K)^{-1} \mathcal{D}_K^{-1}\] (here $\mathcal{D}_K$ is the different ideal of $K$).\\

Now suppose $p\nmid 2d$. As $p \neq 2$ we have $I \otimes \Z_p=(IO_K) \otimes \Z_p$ and $J \otimes \Z_p=(JO_L) \otimes \Z_p$. Moreover, $p\nmid d=\d(K)$ implies that $\mathcal{D}_K \otimes \Z_p= O_K\otimes \Z_p$, thus it follows that each $\alpha_i^{*}\otimes 1$ belongs to $(IO_K)^{-1}\otimes \Z_p$, and since $\phi \otimes 1$ maps $I \otimes \Z_p$ into $J\otimes \Z_p$ we conclude that the right hand side of (\ref{epdel}) belongs to

\[ ((I O_K)^{-1} (JO_L)\otimes \Z_p)^2 \subset (I^{-1} J O_{KL} \otimes \Z_p)^2= \frac{\epsilon}{\delta} O_{KL} \otimes \Z_p .\]
Therefore $c^2\otimes 1 \in O_{KL} \otimes \Z_p$. The case $p\mid d$ is a straightforward adaptation of the proof in \cite[Proposition 4.1]{Casimir}, since in this case $\delta$, $\epsilon$  are prime to $p$ and $p\neq 2$, so $I\otimes \mathbb{Z}_p=R\otimes \mathbb{Z}_p=O_K\otimes \mathbb{Z}_p$ and $J \otimes \Z_p=S \otimes \Z_p=O_L \otimes \Z_p$.\\

It remains to prove (\ref{pint}) when $p=2$. Since $\delta$,$\epsilon$ and $\mathcal{D}_K$ are coprime to $2$, it would be enough to show that $\phi\otimes 1$ maps $O_K\otimes \mathbb{Z}_2$ into $O_L\otimes \mathbb{Z}_2$. Indeed, we can compute the right hand side of $(3)$ using a fixed $\Z$-basis $\langle \alpha_1, \alpha_2, \alpha_2 \rangle$ of $O_K$ (recall that the Casimir invariant is independent of the choice of the $\Q$-basis of $K$). Then $\alpha_i^*\otimes 1 \in O_K \otimes \Z_2$ and $(\phi\otimes1)(\alpha_i \otimes 1) \in O_L \otimes \Z_2$, showing that $c^2$ is $2$-integral.\\

To prove  $(\phi \otimes 1)(O_K \otimes \Z_2) \subset O_L \otimes \Z_2$, we will use the fact that given $z \in S \otimes \Z_2$ we have $z\in 2 (O_L \otimes \Z_2)$ if and only if $\tr(z^2)\equiv \tr(z) \equiv 0 \mod 2$ (see the first remark before the proof). Let $x \in O_K$, then $2x \otimes 1 \in R \otimes \Z_2=I \otimes \Z_2$ and so $y:=\phi(2x) \otimes 1=(\phi\otimes 1)(2x \otimes 1) \in J\otimes \Z_2=S \otimes \Z_2$. Thus $z:=\frac{\phi(2x)^2}{\epsilon}\otimes 1=\frac{y^2}{\epsilon \otimes 1} \in S \otimes \Z_2$, moreover,

\[\tr(z)=\tr\left(\frac{\phi(2x)^2}{\epsilon}\otimes 1\right)=\tr\left(\frac{(2x)^2}{\delta} \otimes 1\right)= 4\, \tr\left(\frac{x^2}{\delta} \otimes 1\right)\equiv 0 \mod 4 \mathbb{Z}_2,\]

\noindent because $\frac{x^2}{\delta} \otimes 1 \in O_K \otimes \Z_2$. In particular, $\tr(z) \equiv 0 \mod 2$, and as we already knew that $z \in S \otimes \Z_2$, it follows that $z \in 2(O_L \otimes \Z_2)$. Hence  $y^2 \in  2(O_L\otimes \mathbb{Z}_2)$,  which implies  $y=2(\phi\otimes 1)(x \otimes 1) \in 2(O_L\otimes \mathbb{Z}_2)$, because $\tr(y)\equiv \tr(y^2) \equiv 0 \mod 2$ and $y \in S \otimes \Z_2$. Hence  $(\phi \otimes 1)( x \otimes 1) \in O_L\otimes \mathbb{Z}_2$. This finishes the proof of $(2)$.\\

Next we have two cases:
\begin{itemize}
\item If $K \not \cong L$, by \cite[Corollary 3.3]{Casimir}, $K$ and $L$ are linearly disjoint, and if $\{\sigma_1,\sigma_2,\sigma_3\}$ and $\{\tau_1,\tau_2,\tau_3\}$ are the embeddings of $K$ and $L$, respectively, with $\sigma_1$ and $\tau_1$ the inclusions $\iota_K$ and $\iota_L$; then for each $1\leq i,j\leq 3$ there exists a unique embedding $\theta_{ij}:KL\hookrightarrow \mathbb{R}$ extending both $\sigma_i$ and $\tau_j$. \\

Let $c_{ij}:=\langle \sigma_i,\tau_j\phi \rangle_{\tr_{K/\mathbb{Q}}} \sqrt{\frac{\sigma_i(\delta)}{\tau_j(\epsilon)}} \in \mathbb{R}$, since $\phi$ is an isometry the matrix $U=(c_{ij})$ must be orthogonal. Indeed, let $\langle \alpha_1,\alpha_2,\alpha_3\rangle$ be any $\Q$-basis of $K$ with dual basis $\langle \alpha_1^*,\alpha_2^*\alpha_3^*\rangle$ with respect to the trace form. Let $\beta_i=\phi(\alpha_i)$, $i=1,2,3$; $A=(\sigma_{j}(\alpha_i))$, $A':=(\sigma_{i}(\alpha_j^*))$, $B=(\tau_j(\beta_i))$, $\Delta:={\rm diag}\left(\frac{1}{\sqrt{\sigma_1(\delta)}},\frac{1}{\sqrt{\sigma_2(\delta)}},\frac{1}{\sqrt{\sigma_3(\delta)}}\right)$ and $\mathcal{E}:={\rm diag}\left( \frac{1}{\sqrt{\tau_1(\epsilon)}},\frac{1}{\sqrt{\tau_2(\epsilon)}},\frac{1}{\sqrt{\tau_3(\epsilon)}}\right)$. Then,

\[A \Delta ^2 A^t=\left(\tr\left(\frac{\alpha_i \alpha_j}{\delta}\right)\right)=\left(\tr\left(\frac{\beta_i \beta_j}{\epsilon}\right) \right)=B \mathcal{E}^2 B^t,\]

shows that $\Delta^{-1} A^{-1}B \mathcal{E}=\Delta^{-1} A'B \mathcal{E}=U$ is orthogonal.\\

However, $\theta_{ij}(c_{11}^2)=c_{ij}^2\leq 1$, so $c_{11}^2$ is a positive real algebraic integer all whose conjugates are bounded by $1$ and thus $c_{11}^2 \in \{0, 1\}$. This contradicts that $U=(c_{ij})=(c_{11})$ is nonsingular.

\item If $K \cong L$, then $O_K \cong O_L$. So the integral ternary quadratic forms defining the quartic rings form which $(R,I,\delta)$ and $(s,J,\epsilon)$ come from have equivalent cubic resolvent forms $f$, hence the corresponding cubic forms $F=2f$ are equivalent  and thus $R \cong S$. By changing $S,J$ and $\epsilon$ by their images in $R$ under this isomorphism if necessary, we may assume $R=S$.

 Let 
$c_{ij}:=\langle \sigma_i,\sigma_j\phi \rangle_{\tr_{K/\mathbb{Q}}} \sqrt{\frac{\sigma_i(\delta)}{\sigma_j(\epsilon)}}\in \mathbb{R}$
, as before we have that $U=(c_{ij})$ is orthogonal and $c_{ij}^2\leq 1$  for all $i,j$. Now let $\widetilde{K}$ be the Galois closure of $K$, for every $\sigma \in \textrm{Gal}(\widetilde{K}/\mathbb{Q})$ and $i,j$ we have that 
$$\sigma(c_{ij}^2)=c_{i'j'}^2\leq 1$$
for some $i',j'$, thus here again we find $c_{ij}^2 \in \{0,1\}$, moreover, since $U$ is orthogonal exactly one of the $c_{ij}^2$ is $1$ on each column and  row of $U$ and from the relation
$$\sigma_i\phi=\sum_{j} \langle \sigma_j,\sigma_i\phi \rangle_{\tr_{K/\mathbb{Q}}} \sigma_i $$
follows that $c_{ij}^2=\delta_{ij}$ (Kronecker delta). In particular, if $\kappa=\langle \sigma_1,\sigma_1\phi \rangle_{\tr_{K/\mathbb{Q}}}\in K$
$$1=c_{11}^2=\kappa^2\;\frac{\delta}{\epsilon}$$
hence $\epsilon=\kappa^2 \delta$ and, as $\phi(x)=\kappa x$, $J=\kappa I$. Therefore, the triples $(R,I,\delta)$ and $(R,J,\epsilon)$ are equivalent.
\end{itemize}

\end{proof}

\section{Acknowledgments}

We would like to thank the referee for the careful reading of the paper, and for their helpful comments.

\noindent
{\footnotesize Carlos Rivera-Guaca, Department of Mathematics, University of Washington,
Seattle, USA. ({\tt caariv@uw.edu})}

\noindent
{\footnotesize Guillermo Mantilla-Soler, Department of Mathematics, Universidad Konrad Lorenz,
Bogot\'a, Colombia. Department of Mathematics and Systems Analysis, Aalto University, Espoo, Finland. ({\tt gmantelia@gmail.com})}

\end{document}